\documentclass[a4paper]{amsart}
\usepackage{amsthm,amsmath,amssymb}
\usepackage{enumitem}
\usepackage{titlesec}
\usepackage{mathrsfs}
\usepackage{mathtools}
\usepackage{comment}
\usepackage{tcolorbox}
\titleformat{\section}{\normalfont\scshape\centering}{\thesection}{1em}{}
\titleformat{\subsection}{\bfseries}{\thesubsection}{1em}{}

\newtheorem{theorem}{Theorem}[section]

\newtheorem{lemma}{Lemma}[section]
\newtheorem{prop}{Proposition}[section]

\theoremstyle{remark}
\newtheorem{rem}{Remark}
\renewcommand{\Re}{\operatorname{Re}}
\renewcommand{\Im}{\operatorname{Im}}

\makeatletter
\@namedef{subjclassname@2020}{%
  \textup{2020} Mathematics Subject Classification}
\makeatother
\title[Hardy's $Z$-function associated with Selberg class]
{On Hardy's $Z$-function and its derivatives associated with the extended Selberg class}
\author{Hirotaka Kobayashi}
\date{}
\address{National Fisheries University, 2-7-1, Nagatahon-machi, Shimonoseki-shi, Yamaguchi 759-6595, Japan}
\email{h.kobayashi@fish-u.ac.jp}
\begin{document}

\begin{abstract}
    Hardy's $Z$-function $Z(t)$ is a real-valued function of the real variable $t$,
    and whose zeros correspond exactly to the zeros of the Riemann zeta-function on the critical line.
    In 2012, K.~Matsuoka showed that for every non-negative integer $k$,
    there exists a $T=T(k)>0$ such that $Z^{(k+1)}(t)$ has exactly one zero between consecutive zeros of $Z^{(k)}(t)$ for $t\ge T$ under the Riemann Hypothesis.
    In this paper, we extend Matsuoka's theorem to  $L$-functions in extended Selberg class.
\end{abstract}

\maketitle


\section{Introduction}

The purpose of this paper is to extend several results concerning Hardy's $Z$-function and its derivatives to generalised Hardy's $Z$-functions associated with $L$-functions in the extended Selberg class.
Let $s=\sigma+it$ be the complex variable.
Hardy's $Z$-function associated with the Riemann zeta-function $\zeta(s)$ is defined by
\[
Z(t)=e^{i\theta(t)}\zeta \left( \frac{1}{2}+it \right)
=\left( \pi^{-it}\frac{\Gamma(\frac{1}{4}+\frac{it}{2})}{\Gamma(\frac{1}{4}-\frac{it}{2})} \right)^{1/2}\zeta \left(\frac{1}{2}+it\right).
\]
We have $|Z(t)|=|\zeta(1/2+it)|$.
From the functional equation of the Riemann zeta-function,
it follows that $Z(t)$ is real-valued.
Therefore, zeros of $Z(t)$ coincide with those of $\zeta(s)$ on the critical line.
Accordingly, it is important to investigate the behaviour of $Z(t)$ and its derivatives.
Since many general $L$-functions are expected to satisfy analogues of the Riemann hypothesis (RH), it is natural to consider analogues of Hardy's $Z$-function for more general $L$-functions.

K.~Matsumoto and Y.~Tanigawa \cite{M-T} construct a meromorphic function $\eta_{k}(s)$,
whose zeros on the critical line coincide with those of $Z^{(k)}(t)$.
They prove that the number of zeros (counted with multiplicity as in what follows) of $\eta_{k}(s)$ in the rectangle $\{ s=\sigma+it \mid 1-2m<\sigma<2m,\ 0<t<T \}$ is
\[
\frac{T}{2\pi}\log \frac{T}{2\pi}-\frac{T}{2\pi}+O_{k}(\log T),
\]
where $m$ is a sufficiently large positive integer, and the index $k$ in the error term means that the implied constant depends on $k$.
Moreover, under the assumption of the Riemann hypothesis, except for finitely many zeros, zeros of $\eta_{k}(s)$ in the above rectangle are on the critical line $\sigma=1/2$.
This shows that RH for the Riemann zeta-function implies an analogue of RH for $\eta_k(s)$.

Later, K.~Matsuoka shows in his unpublished work \cite{M} that, assuming RH, for every non-negative integer $k$,
there exists a $T=T(k)>0$ such that $Z^{(k+1)}(t)$ has exactly one zero between consecutive zeros of $Z^{(k)}(t)$ for $t\ge T$.
For convenience, we refer to this result as {\it the Interlacing Theorem}.
In \cite{K}, the present author simplifies his proof by constructing an entire function $\xi_{k}(s)$ associated with $Z^{(k)}(t)$.

Two well-known classes of $L$-functions that include the Riemann zeta-function as a member are the Selberg class $\mathcal{S}$ and the extended Selberg class $\mathcal{S}^{\sharp}$.
As A.~Ivi\'{c} \cite[p.~51]{I} pointed out, the analogue of Hardy's $Z$-function for some $L$-functions in the Selberg class may be defined.
Indeed, A.~Mukhopadhyay, K.~Srinivas, and K.~Rajkumar \cite{M-S-R} introduce the analogue of Hardy's $Z$-function for functions $F$ in the Selberg class to show that under some suitable conditions, such functions $F$ have infinitely many zeros on the critical line.
R.~\v{S}im\.{e}nas \cite{Si} studies the distribution of zeros of the $k$th derivatives of $L$-functions in the extended Selberg class and obtains the zero-free regions for these derivatives and a Riemann--von Mangoldt-type estimate of the count of their nontrivial zeros.
Recently, S.~Chaubey, S.~S.~Khurana, and A.~I.~Suriajaya \cite{C-K-S} study the relation between RH for an $L$-function $F$ in the Selberg class and the distribution of zeros of derivatives of $F$.

The Interlacing Theorem implies that the existence of a positive local minimum or a negative local maximum of $Z^{(k)}(t)$ for sufficiently large $t$ would contradict RH.
Therefore, the Interlacing Theorem may provide useful insight into structural properties related to the validity of RH,
and it is natural to investigate which zeta- and $L$-functions satisfy analogous interlacing properties.
By generalising the construction of $\xi_k(s)$ in \cite{K}, the present paper extends the results of Matsumoto and Tanigawa and of Matsuoka from the Riemann zeta-function to $L$-functions in the extended Selberg class.

For $F \in \mathcal{S}^{\sharp}$, we denote its degree by $d_F$ and define a generalised Hardy's $Z$-function $Z_{F}(t)$ via the functional equation satisfied by $F$.
Precise definitions are recalled in Section 2.
We prove the following generalisation of 
the Interlacing Theorem to the extended Selberg class:

\begin{theorem}[Interlacing Theorem]\label{th_main}
  If RH for $F(s)$ is true, then for every non-negative integer $k$,
  there exists a $T=T(k)>0$ such that $Z_{F}^{(k+1)}(t)$ has exactly one zero between consecutive zeros of $Z_{F}^{(k)}(t)$ for $t\ge T$.
\end{theorem}

We note that the Euler product condition is not required in the proof.
To prove this theorem, we introduce a meromorphic function $F_{k}(s)$ for each non-negative integer $k$, satisfying
\[
|Z_{F}^{(k)}(t)|=\left| F_{k}\left( \frac{1}{2}+it \right) \right|.
\]
The proof of Theorem~\ref{th_main} relies on the following three results.
Let $N(T;F_{k})$ be the number of zeros of $F_{k}(s)$ in the rectangle
\[
\mathscr{R}=\{ s=\sigma+it \mid 1-\sigma_{F,k}<\sigma<\sigma_{F,k}, \ 0<t<T \},
\]
where $\sigma_{F,k}$ is a sufficiently large positive number.
\begin{theorem}\label{th_countingzeros}
  For every non-negative integer $k$,
  \[
  N(T;F_{k})=\frac{d_{F}}{2\pi}T\log T+c_{F}T+O_{k}(\log T),
  \]
  where $c_{F}$ depends only on $F$.
\end{theorem}

Furthermore, $F_k(s)$ satisfies an analogue of RH:
\begin{theorem}\label{th_RH}
  Under the assumption of RH for $F(s)$,
  all non-real zeros of $F_k(s)$ in the region $\mathscr{R}$,
  with at most finitely many exceptions,
  lie on the critical line.
\end{theorem}

Theorems~\ref{th_countingzeros} and~\ref{th_RH} are analogues of Matsumoto and Tanigawa's theorems.
Theorem~\ref{th_RH} plays an important role in the proof of the following theorem.

\begin{theorem}\label{th_EMF}
  Under the assumption of RH for $F(s)$, we have
  \[
  \frac{d}{dt}\frac{Z_{F}^{(k+1)}}{Z_{F}^{(k)}}(t)=-\sum_{\gamma_{k}}\frac{1}{(t-\gamma_{k})^2}+O_{k}(t^{-1}),
  \]
  where $\gamma_{k}$ runs over the zeros of $Z_{F}^{(k)}(t)$, counted with multiplicity.
\end{theorem}

Sections 2--5 are devoted to several auxiliary results needed to establish Theorems~\ref{th_countingzeros}--\ref{th_EMF}.
Theorems~\ref{th_countingzeros} and~\ref{th_RH} are proved in Sections 6 and 7, respectively.
In the final section, we prove Theorem~\ref{th_EMF} and then show that Theorems~\ref{th_countingzeros} and~\ref{th_EMF} together imply Theorem~\ref{th_main}.

\section{Definition and basic properties of $F_{k}(s)$}

The Selberg class $\mathcal{S}$, introduced by A.~Selberg \cite{S} in 1992, consists of Dirichlet series $F(s)$ satisfying the following five axioms:

\begin{enumerate}[label=(S\arabic*)]
	\item It can be expressed as a Dirichlet series
	\[
	F(s)=\sum_{n=1}^{\infty}\frac{a_{F}(n)}{n^s},
	\]
	which is absolutely convergent if $\sigma>1$ with $a_{F}(1)=1$.
	\item There exists an integer $m\ge 0$ such that $(s-1)^mF(s)$ extends to an entire function of finite order.
	The smallest $m$ is denoted by $m_{F}$.
	\item There exist an integer $r\ge 0$, real numbers $Q>0$, $\lambda_{j}>0$, complex numbers $\mu_{j}$ with $\Re \mu_{j}\ge 0$, and $\omega \in \mathbb{C}$ with $|\omega|=1$,
	such that the function $\xi_{F}(s)$ defined by
	\begin{align*}
		\xi_{F}(s)
		&=s^{m_{F}}(s-1)^{m_{F}}Q^{s}\prod_{j=1}^{r}\Gamma(\lambda_{j}s+\mu_{j})F(s) \\
		&=s^{m_{F}}(s-1)^{m_{F}}\gamma(s)F(s)
	\end{align*}
	satisfies the functional equation
	\[
	\xi_{F}(s)=\omega\overline{\xi_{F}(1-\overline{s})},
	\]
	where $\Gamma(s)$ is the gamma function.
	\item For every $\varepsilon>0$, $a_{F}(n)\ll_{\varepsilon}n^{\varepsilon}$.
	\item For every sufficiently large $\sigma$,
	\[
	\log F(s)=\sum_{n=1}^{\infty}\frac{b_{F}(n)}{n^{s}},
	\]
	where $b_{F}(n)=0$ unless $n=p^m$ with $m\ge 1$, and $b_{F}(n)\ll n^{\theta_{F}}$ for some $\theta_{F}<1/2$.
\end{enumerate}

We define the degree of $F$ by
\[
d_{F}=2\sum_{j=1}^{r}\lambda_{j}.
\]
This quantity depends only on $F$.
It is known that if $F \in \mathcal{S}$, then either $F=1$ or $d_{F}\ge 1$ (\cite{C-G}).
We refer to $\gamma(s)$ in (S3) as the gamma factor.
By (S3) and (S5), $F\in \mathcal{S}$ has no zeros outside the critical strip $0\le \sigma\le 1$ except for zeros in the left half-plane $\sigma\le 0$ given by the poles of the gamma factor.
The extended Selberg class $\mathcal{S}^{\sharp}$ consists of Dirichlet series satisfying axioms (S1)--(S3).

Let $H(s)=\omega\overline{\gamma(1-\overline{s})}/\gamma(s)$.
We see that
\[
F(s)=H(s)\overline{F(1-\overline{s})}, \ H(s)\overline{H(1-\overline{s})}=1.
\]
Since
\[
\left|H\left(\frac{1}{2}+it\right)\right|=1,
\]
there exists a continuous real-valued function $\theta_F(t)$ such that
\[
H\left(\frac12+it\right)=e^{-2i\theta_F(t)}.
\]
Since $H(s)$ is meromorphic and non-vanishing on the critical line,
$\theta_F(t)$ may be chosen to be smooth.
Then we define Hardy's $Z$-function for $F(s)$ as
\[
Z_{F}(t):=e^{i\theta_F(t)}F\left( \frac{1}{2}+it \right) \quad (t\in \mathbb{R}).
\]
It follows that $|Z_{F}(t)|=|F(1/2+it)|$ and that $Z_{F}(t)$ is real-valued.

\begin{rem}
	If the coefficients $a_{F}(n)$ are real, then $\overline{F(1-\overline{s})}=F(1-s)$, so that
	\[
	F(s)=H(s)F(1-s), \ H(s)H(1-s)=1.
	\]
	Moreover, if $\omega$ and the $\mu_{j}$ are real, i.e., $\overline{H(s)}=H(\overline{s})$,
	then $Z_{F}(t)$ is an even function (see also \cite[p. 51]{I}).
\end{rem}

Let $\psi_{F}(s)=(H'/H)(s)$.
It then follows that
\begin{equation}\label{psitheta}
\psi_{F} \left(\frac{1}{2}+it \right)=-2\theta_{F}'(t).
\end{equation}
Setting $F_0(s)=F(s)$, we define $F_{k}(s)$ for $k\geq 0$ inductively by
\begin{equation}\label{def_F_k}
F_{k+1}(s)=F_{k}'(s)-\frac{1}{2}\psi_{F}(s)F_{k}(s) \quad (k\geq 0).
\end{equation}

\begin{rem}
	The functions $F_k(s)$ may be regarded as analogues, for
	$F\in\mathcal S^\sharp$, of the functions $f_k(s)$ introduced by
	Matsuoka \cite{M} (denoted by $Z_k(s)$ in \cite{K}).
	
	They are also closely related to the functions $\eta_k(s)$ of
	Matsumoto and Tanigawa \cite{M-T}, which were introduced in the study
	of derivatives of Hardy's $Z$-function for the Riemann zeta-function.
\end{rem}

\begin{prop}\label{ZandF}
For every non-negative $k$, we have
\begin{equation*}
Z_{F}^{(k)}(t)=i^{k}F_{k}\left(\frac{1}{2}+it \right)e^{i\theta_{F}(t)}.
\end{equation*}
\end{prop}

\begin{proof}
The case $k=0$ follows immediately from the definition of $Z_{F}(t)$.
If we assume that the equation is true for $k$, then 
\begin{equation*}
Z_{F}^{(k+1)}(t)=i^{k+1}e^{i\theta_{F}(t)}\left(F_{k}'\left(\frac{1}{2}+it\right)+\theta_{F}'(t)F_{k}\left(\frac{1}{2}+it \right)\right).
\end{equation*}
By \eqref{psitheta} and \eqref{def_F_k}, we conclude that the identity holds for $k+1$.
\end{proof}

\begin{prop}[The Functional Equation]\label{FE}
For every non-negative $k$, we have
\begin{equation}\label{eqFE}
F_{k}(s)=(-1)^k H(s)\overline{F_{k}(1-\overline{s})}.
\end{equation}
\end{prop}

\begin{proof}
The case $k=0$ follows directly from axiom (S3).
If we assume that the equation is true for $k$, then by the definition of $F_{k}(s)$,
\begin{align*}
F_{k+1}(s) &= F_{k}'(s)-\frac{1}{2}\psi_{F}(s)F_{k}(s) \\
&= (-1)^{k} \left( H'(s)\overline{F_{k}(1-\overline{s})}-H(s)\overline{F_{k}'(1-\overline{s})}-\frac{1}{2}\psi_{F}(s)H(s)\overline{F_{k}(1-\overline{s})} \right) \\
&= (-1)^{k+1}H(s)\left( \overline{F_{k}'(1-\overline{s})}+\frac{1}{2}\psi_{F}(s)\overline{F_{k}(1-\overline{s})}-\psi_{F}(s)\overline{F_{k}(1-\overline{s})} \right) \\
&= (-1)^{k+1}H(s)\left( \overline{F_{k}'(1-\overline{s})}-\frac{1}{2}\psi_{F}(s)\overline{F_{k}(1-\overline{s})} \right) \\
&= (-1)^{k+1}H(s)\overline{F_{k+1}(1-\overline{s})}.
\end{align*}
This completes the proof.
\end{proof}

For later use, we derive a more explicit expression for $F_{k}(s)$ for our purposes.
Let $f_0(s)=1$, and define $f_k(s)$ for $k\geq 0$ inductively by
\begin{equation*}
f_{k+1}(s)=f_k'(s)-\frac{1}{2}\psi_{F}(s)f_k(s) \quad (k\geq 0).
\end{equation*}
Then we have the following proposition.
\begin{prop}\label{exF_{k}}
For every non-negative $k$, we have
\begin{equation*}
F_{k}(s)=\sum_{j=0}^{k}\binom{k}{j}f_{k-j}(s)F^{(j)}(s).
\end{equation*}
\end{prop}

\begin{proof}
The case $k=0$ is immediate.
We assume that the identity holds for some non-negative integer $k$.
By the definition,
\begin{align*}
F_{k+1}(s)&=F_{k}'(s)-\frac{1}{2}\psi_{F}(s)F_{k}(s)\\
&=\sum_{j=0}^{k}\binom{k}{j}f_{k-j}'(s)F^{(j)}(s)+\sum_{j=0}^{k}\binom{k}{j}f_{k-j}(s)F^{(j+1)}(s) \\
&\quad-\frac{1}{2}\psi_{F}(s)\sum_{j=0}^{k}\binom{k}{j}f_{k-j}(s)F^{(j)}(s)\\
&=\sum_{j=0}^{k}\binom{k}{j}f_{k+1-j}(s)F^{(j)}(s)+\sum_{j=0}^{k}\binom{k}{j}f_{k-j}(s)F^{(j+1)}(s)\\
&=f_{k+1}(s)F(s)+\sum_{j=1}^{k}\left\{\binom{k}{j}+\binom{k}{j-1}\right\}f_{k+1-j}(s)F^{(j)}(s)+F^{(k+1)}(s)\\
&=f_{k+1}(s)F(s)+\sum_{j=1}^{k}\binom{k+1}{j}f_{k+1-j}(s)F^{(j)}(s)+F^{(k+1)}(s) \\
&=\sum_{j=0}^{k+1}\binom{k+1}{j}f_{k+1-j}(s)F^{(j)}(s).
\end{align*}
Here, to obtain the last equality, we use the relation
\begin{equation*}
\binom{k}{j}+\binom{k}{j-1}=\binom{k+1}{j}.
\end{equation*}
\end{proof}

\section{Some estimates on $\psi_{F}(s)$}
In this section, we establish some estimates for $\psi_{F}(s)$.
From the definition of $H(s)$, we have

\begin{equation*}
\psi_{F}(s)=-2\log Q -\sum_{j=1}^{r}\lambda_{j}\left( \frac{\Gamma'}{\Gamma}(\lambda_{j}(1-s)+\overline{\mu_{j}})+\frac{\Gamma'}{\Gamma}(\lambda_{j}s+\mu_{j}) \right).
\end{equation*}
When $d_{F}=0$, clearly
\begin{equation*}
	\psi_{F}(s)=-2\log Q \ \text{and}\ \psi_{F}^{(k)}(s)=0 \quad (k\ge 1).
\end{equation*}
Let $d_{F}>0$.
By Euler's reflection formula, for $1\le j\le r$, we obtain

\[
\frac{\Gamma'}{\Gamma}(\lambda_{j}(1-s)+\overline{\mu_{j}})=\frac{\Gamma'}{\Gamma}(\lambda_{j}s+1-\lambda_{j}-\overline{\mu_{j}})-\pi\cot \pi(\lambda_{j}s+1-\lambda_{j}-\overline{\mu_{j}}).
\]
Let $\mathscr{D}$ denote the subset of $\mathbb{C}$ obtained by removing 
open discs with radii depending on $k$ centred at 
$s=1+\frac{\overline{\mu_j}+n}{\lambda_j}$ and $s=-\frac{\mu_j+n}{\lambda_j}$ 
for $1\leq j\leq r$ and $n\in\mathbb{Z}_{\geq 0}$.
We write $\mathscr{D}_1 = \mathbb{C} \setminus \mathscr{D}$ for the complement of $\mathscr{D}$.
By Stirling's formula, for $\sigma>\frac{1}{4}$ and $1\le j\le r$, we obtain
\[
\frac{\Gamma'}{\Gamma}(\lambda_{j}s+\mu_{j})=\log (\lambda_{j}s+\mu_{j}) -\frac{1}{2(\lambda_{j}s+\mu_{j})}+O\left(\frac{1}{|s|^2}\right)
\]
and
\[ \frac{d^k}{ds^k}\frac{\Gamma'}{\Gamma}(\lambda_{j}s+\mu_{j})=(-1)^{k-1}(k-1)!\left(\frac{\lambda_{j}}{\lambda_{j}s+\mu_{j}}\right)^{k}+O_{k}(|s|^{-k-1}).
\]
In the same manner, for $1\le j\le r$ and either $\sigma>1+\frac{\Re\mu_{j}-1}{\lambda_{j}}$ or $|t+\frac{\Im \mu_{j}}{\lambda_{j}}|\ge 1$, we have 
\[
\frac{\Gamma'}{\Gamma}(\lambda_{j}s+1-\lambda_{j}-\overline{\mu_{j}})=\log (\lambda_{j}s+1-\lambda_{j}-\overline{\mu_{j}}) -\frac{1}{2(\lambda_{j}s+1-\lambda_{j}-\overline{\mu_{j}})}+O\left(\frac{1}{|s|^2}\right)
\]
and
\[ \frac{d^k}{ds^k}\frac{\Gamma'}{\Gamma}(\lambda_{j}s+1-\lambda_{j}-\overline{\mu_{j}})=(-1)^{k-1}(k-1)!\left(\frac{\lambda_{j}}{\lambda_{j}s+1-\lambda_{j}-\overline{\mu_{j}}}\right)^{k}+O_{k}(|s|^{-k-1}).
\]
Hence we find that there is an absolute positive constant $\sigma_1$ such that
\begin{equation*}
	\Re\frac{\Gamma'}{\Gamma}(\lambda_{j}s+\mu_{j}) \ge \frac{1}{2}\log \sigma \ \text{and} \ \Re\frac{\Gamma'}{\Gamma}(\lambda_{j}s+1-\lambda_{j}-\overline{\mu_{j}}) \ge \frac{1}{2}\log \sigma
\end{equation*}
for $\sigma\ge \sigma_1$.
In the region $\mathscr{D}$, we have for $1\le j\le r$,
\begin{equation*}
\cot \pi(\lambda_{j}s+1-\lambda_{j}-\overline{\mu_{j}})=\begin{cases}
           -i+O(e^{-2\pi \lambda_{j}t}) & \text{$(t\ge 0)$}, \\
           i+O(e^{2\pi \lambda_{j}t}) & \text{$(t<0)$},
           \end{cases}
\end{equation*}
and
\begin{equation*}
\frac{d^k}{ds^k}\cot \pi(\lambda_{j}s+1-\lambda_{j}-\overline{\mu_{j}})=O_k(e^{-2\pi \lambda_{j}|t|}) \quad(k\geq 1).
\end{equation*}

From the above argument, we have
\begin{equation}\label{realpsi}
\begin{split}
	\Re \psi_{F}(s) &\leq -2\log Q +O(1)-\frac{d_{F}}{2}\log \sigma \\
	&\leq -\frac{d_{F}}{4}\log \sigma-2\log Q
\end{split}
\end{equation}
for $s\in \mathscr{D}(\sigma_1) := \{s \in \mathscr{D} \mid \sigma \geq \sigma_1 \}$,
provided that $\sigma_1$ is sufficiently large.
We close this section with the following key lemma.

\begin{lemma}\label{estpsi}
Let $s \in \mathscr{D}$ and $d_{F}>0$.
For sufficiently large $|s|$ with $\sigma>1/2$, we have
\begin{equation*}
\psi_{F}(s)=-\log \lambda Q^{2}s+\frac{d_{F}-4\Theta}{2s}+Ci+O\left(\frac{1}{|s|^2}+e^{-C'|t|}\right),
\end{equation*}
and
\begin{equation*}
\psi_{F}^{(k)}(s)=O_k\left(|s|^{-k}+e^{-C'|t|}\right) \quad(k\geq 1),
\end{equation*}
where $\lambda=\prod_{j=1}^{r}\lambda_{j}^{2\lambda_{j}}$, $\Theta=\sum_{j=1}^{r}\Im \mu_{j}$, $C$ is a real constant depending only on $F$ and the sign of $t$, and $C'$ is a positive constant depending only on $F$.
\end{lemma}
In this lemma, $s$ is limited in the right half plan.
However, for $s$ in the left half plane, we may apply this lemma to $\overline{\psi_{F}(1-\overline{s})}=\psi_{F}(s)$.

\section{Poles of $f_k(s)$ and $F_{k}(s)$}
In this section, we investigate the poles of $f_{k}(s)$ and $F_{k}(s)$.
We begin with the following lemma.

\begin{lemma}
   Let $1\le j\le r$, $n\in \mathbb{Z}_{\ge 0}$ and $d_{F}>0$.
The poles of $\psi_{F}(s)$ are all simple, and located at $s=1+\frac{\overline{\mu_{j}}+n}{\lambda_{j}}$ with residue $-1$ and at $s=-\frac{\mu_{j}+n}{\lambda_{j}}$ with residue $1$.
\end{lemma}

\begin{proof}
Since $H(s)=\omega\overline{\gamma(1-\overline{s})}/\gamma(s)=\omega Q^{1-2s}\prod_{j=1}^{r}\Gamma(\lambda_{j}(1-s)+\overline{\mu_{j}})/\Gamma(\lambda_{j}s+\mu_{j})$,
the zeros of $H(s)$ are located at $s=-\frac{\mu_{j}+n}{\lambda_{j}}$ and the poles at $s=1+\frac{\overline{\mu_{j}}+n}{\lambda_{j}}$, all of which are simple.
The lemma follows.
\end{proof}

\begin{lemma}\label{poles_f_k}
  Let $1\le j\le r$, $n\in \mathbb{Z}_{\ge 0}$ and $d_{F}>0$.
For $k\geq 0$, the function $f_{k}(s)$ has poles of order $k$ which are located only at $s=-\frac{\mu_{j}+n}{\lambda_{j}}, 1+\frac{\overline{\mu_{j}}+n}{\lambda_{j}}$.
\end{lemma}

\begin{proof}
  The case $k=1$ follows immediately from the previous lemma.
  Suppose the lemma holds for some intger $k \ge 1$.
  Let $a$ be a pole of $f_{k}(s)$.
  Then by Laurent expansion at centre $a$, we have
\begin{equation*}
f_{k}(s)=\frac{c_k}{(s-a)^k}+\cdots,
\end{equation*}
where $c_k$ does not vanish.
By the definition and the previous lemma, we have
\begin{equation*}
f_{k+1}(s)=\frac{-kc_k+\frac{c_k}{2}}{(s-a)^{k+1}}+\cdots.
\end{equation*}
Since $-kc_k+c_k/2\neq 0$, the pole at $s=a$ has order $k+1$, which proves the inductive step.
\end{proof}

Combining this lemma with Proposition~\ref{exF_{k}}, we immediately obtain the following.

\begin{lemma}\label{poles_F_k}
  Let $1\le j\le r$, $n\in \mathbb{Z}_{\ge 0}$ and $d_{F}>0$.
For $k\geq 0$, the function $F_{k}(s)$ has poles of order $k$ located at $s=1+\frac{\overline{\mu_{j}}+n}{\lambda_{j}}$ and those of order $k-1$ located at $s=-\frac{\mu_{j}+n}{\lambda_{j}}$.
Moreover, if $F_{0}(s)=F(s)$ has a pole of order $m_{F}$ located at $s=1$ then $F_{k}(s)$ also has a pole at $s=1$ and the order is $m_{F}$.
\end{lemma}
Here, a pole of order $-1$ is understood to mean a zero of order $1$.
We now define $\xi_{F,k}(s)$ by
\[
\xi_{F,k}(s)=s^{m_{F}}(s-1)^{m_{F}}Q^{s}\prod_{j=1}^{r}\Gamma(\lambda_{j}s+\mu_{j})^{-k+1}\Gamma(\lambda_{j}(1-s)+\overline{\mu_{j}})^{-k}F_{k}(s).
\]
It follows from Lemma~\ref{poles_F_k} that $\xi_{F,k}$ is entire.
When $k=0$, this function coincides with $\xi_{F}(s)$ as defined in axiom~(S3).
The following proposition is a direct consequence of Proposition~\ref{FE}.
\begin{prop}
	For each $k\ge 0$,
	\begin{equation}\label{xi_FE}
		\xi_{F,k}(s)=(-1)^{k}\overline{\xi_{F,k}(1-\overline{s})}.
	\end{equation}
\end{prop}

\section{Auxiliary results on $f_k(s)$ and $F_{k}(s)$}

The function $f_k(s)$ admits the following explicit expression.
\begin{prop}\label{exf_k}
	For $k\geq 1$, we have
	\begin{equation*}
		f_k(s)=k!\sum_{\substack{a_{1}, \dots a_{k} \in \mathbb{Z}_{\ge 0} \\ a_1+2a_2+\dots+ka_k=k}}\left(-\frac{1}{2} \right)^{a_1+\dots+a_k}\prod_{l=1}^{k}\frac{1}{a_{l}!}\left(\frac{\psi_{F}^{(l-1)}(s)}{l!} \right)^{a_l}.
	\end{equation*}
\end{prop}
The proof follows the same argument as that of Proposition~2.4 in \cite{K}.
By this proposition, we have
\begin{equation}\label{f_k}
f_k(s)=\left(-\frac{\psi_{F}(s)}{2}\right)^k+\Lambda_k(s),
\end{equation}
where
\begin{equation}\label{Lambda}
\Lambda_{k}(s)=k!\sum_{\substack{a_{1}, \dots a_{k} \in \mathbb{Z}_{\ge 0} \\ a_1+2a_2+\dots+ka_k=k \\ a_{1}\le k-1}}\left(-\frac{1}{2} \right)^{a_1+\dots+a_k}\prod_{l=1}^{k}\frac{1}{a_{l}!}\left(\frac{\psi_{F}^{(l-1)}(s)}{l!} \right)^{a_l}.
\end{equation}
Note that the condition $a_1 \leq k-1$ may be replaced by $a_1 \leq k-2$, since $a_1 = k-1$ is incompatible with the constraint $a_1 + 2a_2 + \cdots + ka_k = k$.
We write \eqref{f_k} as
\begin{equation}\label{f_k2}
f_k(s)=\left(-\frac{\psi_{F}(s)}{2}\right)^kA_k(s),
\end{equation}
where 
\begin{equation}\label{Akest}
A_k(s)=1+\frac{\Lambda_k(s)}{(-\frac{\psi_{F}(s)}{2})^k}.
\end{equation}
By Lemma \ref{estpsi} and \eqref{Lambda}, we obtain
\begin{equation}\label{A_kest}
A_k(s)=1+O_k((\log |s|)^{-2}) \quad (k\geq 1),
\end{equation}
and so
\begin{equation}\label{f_kest}
f_k(s)=\left(-\frac{\psi_{F}(s)}{2}\right)^k(1+O_k((\log |s|)^{-2}))
\end{equation}
for $s\in \mathscr{D}$ whose absolute value is sufficiently large when $d_{F}>0$.

We next determine the approximate location of the zeros of $f_k(s)$ for sufficiently large $|s|$.

\begin{lemma}\label{zero_f_k}
  Let $\sigma_{1}$ be sufficiently large number and $d_{F}>0$.
  All zeros of $f_{k}(s)$ with $\sigma \le 1-\sigma_{1}$ or $\sigma_{1} \le \sigma$ are located in $\mathscr{D}_{1}$,
  and the number of those in each circle is $k$.
  Let $T$ be sufficiently large.
  In the region $\{ s\mid 1-\sigma_{1}\le \sigma\le \sigma_{1}, \ |t|>T \}$, there is no zero of $f_{k}(s)$.
\end{lemma}

\begin{proof}
  By \eqref{f_kest}, if $|s|$ is sufficiently large, $\Re f_{k}(s)$ is positive in $\mathscr{D}$.
  Hence, the lemma follows from the argument principle and Lemma~\ref{poles_f_k}.
\end{proof}

By Proposition \ref{exF_{k}} and \eqref{f_k2}, we can write
\begin{equation*}
F_{k}(s)=\left(-\frac{\psi_{F}(s)}{2}\right)^kA_k(s)g_k(s),
\end{equation*}
where
\begin{equation}\label{gk}
g_k(s)=F(s)+\sum_{j=1}^{k-1}\binom{k}{j}\frac{f_{k-j}(s)}{f_k(s)}F^{(j)}(s)+\frac{F^{(k)}(s)}{f_k(s)}.
\end{equation}
Note that $F(s)=1+O(2^{-\sigma})$ for $\sigma > 2$, and $F^{(k)}(s)=O_{k}(2^{-\sigma})$ for $k\ge 1$ and $\sigma > 2$.
Using these estimates and \eqref{gk}, we obtain
\begin{equation*}
g_k(s)=1+O_k(2^{-\sigma}) \quad (k\geq 1),
\end{equation*}
and so, with \eqref{A_kest},
\begin{equation}\label{Fkest}
F_{k}(s)=\left(-\frac{\psi_{F}(s)}{2}\right)^k\{1+O_k((\log \sigma)^{-2}) \} \quad(k\geq 1)
\end{equation}
for $s\in \mathscr{D}(\sigma_1)$ with sufficiently large $\sigma_{1}$.

We similarly determine the approximate location of the zeros of $F_k(s)$ for sufficiently large $|\sigma|$.

\begin{lemma}
  Let $\sigma_{1}$ be sufficiently large number.
  All zeros of $F_{k}(s)$ with $\sigma \le 1-\sigma_{1}$ or $\sigma_{1} \le \sigma$ are located in $\mathscr{D}_{1}$,
  and the number of those in each circle is $k$.
\end{lemma}

\begin{proof}
  For $\sigma_{1}\le \sigma$, by \eqref{Fkest} we observe that $\Re F_{k}(s)$ is positive in $\mathscr{D}$.
  Thus the lemma follows in the same manner as in the proof of Lemma~\ref{zero_f_k}.
  The case $\sigma \leq 1 - \sigma_1$ follows from the functional equation \eqref{eqFE}.
\end{proof}

\section{Proof of Theorem~\ref{th_countingzeros}}

By \eqref{realpsi} and \eqref{Fkest}, we can find positive number $\sigma_{F,k} \in \mathscr{D}$ such that $-\Re \psi_{F}(s)$, $\Re A_k(s)$, and $\Re g_k(s)$ are all positive on the line $\sigma=\sigma_{F,k}$.
Let $\mathscr{R}$ be the rectangle defined in Section~1, with such $\sigma_{F,k}$,
and let $\mathscr{L}$ be the positively oriented boundary of $\mathscr{R}$, indented along lower side by small semicircles above any poles and zeros of $\xi_{F,k}(s)$ lying on the real axis.
We obtain
\begin{equation*}
N(T;F_{k})=\frac{1}{2\pi i}\int_{\mathscr{L}} \frac{\xi_{F,k}'}{\xi_{F,k}}(s)\, ds.
\end{equation*}
The contribution from the lower side of $\mathscr{L}$ is $O_k(1)$.
By the functional equation~\eqref{xi_FE} of $\xi_{F,k}(s)$, we obtain
\begin{equation*}
	N(T;F_{k})=\frac{1}{\pi }\Im \left\{\int_{\sigma_{F,k}}^{\sigma_{F,k}+iT}+\int_{\sigma_{F,k}+iT}^{1/2+iT}\right\}\frac{\xi_{F,k}'}{\xi_{F,k}}(s)\, ds.
\end{equation*}
Therefore
\begin{equation}\label{NF_{k}}
N(T;F_{k})=\frac{\theta_{F}(T)}{\pi}+S(T;F_{k})+O_k(1),
\end{equation}
where
\begin{align*}
	S(T;F_{k})
	&=
	\frac{1}{\pi}\Im \left\{\int_{\sigma_{F,k}}^{\sigma_{F,k}+iT}+\int_{\sigma_{F,k}+iT}^{1/2+iT} \right\}\frac{F_{k}'}{F_{k}}(s)\, ds \\
	&=
	\frac{1}{\pi} \left\{\int_{\sigma_{F,k}}^{\sigma_{F,k}+iT}+\int_{\sigma_{F,k}+iT}^{1/2+iT} \right\}\, d\arg F_{k}(s).
\end{align*}
By Stirling's formula, we can immediately show that
\begin{equation*}
\frac{\theta_{F}(T)}{\pi}=\frac{d_{F}}{2\pi}T\log \frac{T}{2\pi}+c_{F}T+c_{F}'+O(T^{-1}),
\end{equation*}
where $c_{F}$ and $c_{F}'$ are constants depending only on $F$.
By the estimates in the previous section and the choice of $\sigma_{F,k}$, $-\Re \psi_{F}(s), \Re A_k(s)$, and $\Re g_k(s)$ are all positive on the line $\sigma=\sigma_{F,k}$.
Hence the total variation of the argument of each of these functions does not exceed $\pi$, which gives
\begin{equation}\label{argver}
\left|\int_{\sigma_{F,k}}^{\sigma_{F,k}+iT}d\arg F_{k}(s) \right|\leq k\pi+\pi+\pi=(k+2)\pi.
\end{equation}
Suppose that $\Re \psi_{F}(s)$ vanishes $q$ times on the horizontal segment, then
\begin{equation*}
\left|\int_{\sigma_{F,k}+iT}^{1/2+iT}d\arg (-\psi_{F}(s)) \right|\leq(q+1)\pi.
\end{equation*}
Now we should bound $q$, and $q$ can be considered as the number of zeros of the function
\begin{equation*}
\varphi(z)=\frac{1}{2}\{\psi_{F}(z+iT)+\psi_{F}(z-iT)\}
\end{equation*}
for $\Im z=0, 1/2\leq \Re z\leq \sigma_{F,k}$, so that $q\leq n(\sigma_{F,k}-1/2)$,
where $n(r)$ denotes the number of zeros of $\varphi(z)$ in the disc $|z-\sigma_{F,k}|\leq r$.
When $0\leq \alpha<1/2$, we have
\begin{equation*}
\int_{0}^{\sigma_{F,k}-\alpha}\frac{n(r)}{r}dr\geq \int_{\sigma_{F,k}-\frac{1}{2}}^{\sigma_{F,k}-\alpha}\frac{n(r)}{r}dr
\geq n\left(\sigma_{F,k}-\frac{1}{2}\right)\log \frac{\sigma_{F,k}-\alpha}{\sigma_{F,k}-1/2}.
\end{equation*}
On the other hand, Jensen's formula gives
\begin{equation*}
\int_{0}^{\sigma_{F,k}-\alpha}\frac{n(r)}{r}dr=\frac{1}{2\pi}\int_{0}^{2\pi}\log|\varphi(\sigma_{F,k}+\sigma_{F,k}e^{i\theta})|d\theta-\log|\varphi(\sigma_{F,k})|.
\end{equation*}
Therefore we obtain
\begin{equation}\label{qest}
q\ll_k \frac{1}{2\pi}\int_{0}^{2\pi}\log|\varphi(\sigma_{F,k}+\sigma_{F,k}e^{i\theta})|d\theta-\log|\varphi(\sigma_{F,k})|.
\end{equation}
Assume that $T$ is sufficiently large relative to $\sigma_{F,k}$, and $\mathscr{D}(\sigma_{F,k},T)=\{s\mid \sigma \geq 1/2, |t-T|\leq \sigma_{F,k} \}$.
By Lemma \ref{estpsi}, we have
\begin{equation*}
\psi_{F}(s)\ll \log T
\end{equation*}
in $\mathscr{D}(\sigma_{F,k},T)$.
Hence the right-hand side of~\eqref{qest} is $O(\log \log T)$,
and we obtain
\begin{equation*}
\int_{\sigma_{F,k}+iT}^{1/2+iT}d\arg (-\psi_{F}(s))=O(\log \log T).
\end{equation*}
By \eqref{Akest}, we have $A_k(s)=O_k(1)$ in $\mathscr{D}(\sigma_{F,k},T)$,
and there exists a positive constant $c$ such that $F(s)=O(t^{c})$ in the same region,
and so $g_{k}(s)=O_{k}(t^c)$ in the region.
Therefore, by the same argument, we can show that
\begin{equation*}
\int_{\sigma_{F,k}+iT}^{1/2+iT}d\arg A_k(s)=O_k(1)
\end{equation*}
and
\begin{equation*}
\int_{\sigma_{F,k}+iT}^{1/2+iT}d\arg g_k(s)=O_k(\log T).
\end{equation*}
Hence we have
\begin{equation*}
\int_{\sigma_{F,k}+iT}^{1/2+iT}d\arg F_{k}(s)=O_k(\log T)
\end{equation*}
and this implies $S(T;F_{k})=O_k(\log T)$ with \eqref{argver}.
This completes the proof.

\section{Proof of Theorem~\ref{th_RH}}

Applying Littlewood's lemma to $F_{k}(s)$ on the rectangle
$R=\{s=\sigma+it \mid 1/2\leq \sigma \leq \sigma_{F,k}, 0\leq t \leq T \}$, we obtain
\begin{align*}
&\quad\quad \frac{1}{2\pi}\int_{0}^{T}\log F_{k}(1/2+it)\, dt-\frac{1}{2\pi}\int_{0}^{T}\log F_{k}(\sigma_{F,k}+it)\, dt \\
&\quad+\frac{1}{2\pi i}\int_{1/2}^{\sigma_{F,k}}\log F_{k}(\sigma+iT)\, d\sigma-\frac{1}{2\pi i}\int_{1/2}^{\sigma_{F,k}}\log F_{k}(\sigma)\, d\sigma \\
&=\sum \mathrm{dist},
\end{align*}
where $\mathrm{dist}$ denotes the sum of the distances from the line $\sigma=\tfrac{1}{2}$ to each zero of $F_{k}(s)$ inside the rectangle.
Taking imaginary parts, we obtain
\begin{equation}\label{little}
\begin{split}
&\quad \int_{0}^{T}\arg F_{k}(1/2+it)\, dt-\int_{0}^{T}\arg F_{k}(\sigma_{F,k}+it)\, dt \\
&=\int_{1/2}^{\sigma_{F,k}}\log |F_{k}(\sigma+iT)|\, d\sigma-\int_{1/2}^{\sigma_{F,k}}\log |F_{k}(\sigma)|\, d\sigma.
\end{split}
\end{equation}
The second term on the right-hand side is $O_k(1)$.
We recall that $-\Re \psi_{F}(s), \Re A_k(s)$, and $\Re g_k(s)$ are all positive on the line $\sigma=\sigma_{F,k}$ and
\[
F_{k}(s)=\left(-\frac{\psi_{F}(s)}{2}\right)^kA_k(s)g_k(s).
\]
It follows that
\[
|\arg F_k(\sigma_{F,k}+it)|<\frac{\pi}{2}(k+2).
\]
Hence the second term on the left-hand side of~\eqref{little} is $O_k(T)$.
For the first integral on the right-hand side of~\eqref{little}, we have
\begin{align*}
\psi_{F}(s)=O(\log T), A_k(s)=O_k(1), \text{and} \ g_k(s)=O_k(T^c)
\end{align*}
in $\mathscr{D}(\sigma_{F,k},T)$.
Hence the integral is $O_{k}(\log T)$.

Combining these estimates with~\eqref{little}, and recalling that
\[
S(t;F_{k})=\frac{1}{\pi}\arg F_{k}(1/2+it),
\]
where the branch of the argument is chosen consistently with Littlewood's lemma, we obtain
\begin{equation}\label{StF_{k}}
	\int_{0}^{T}S(t;F_{k})\, dt=O_k(T).
\end{equation}
Let $N_{0}(t;F_{k})$ denote the number of zeros of $F_{k}(\tfrac{1}{2}+iu)$ in the interval $(0,t)$.
Then, by Rolle's theorem, we have
\begin{equation*}
N_{0}(t;F_{0})\leq N_0(t;F_{k})+k-1 \quad (k\geq 1).
\end{equation*}
It follows from~\eqref{NF_{k}} that
\begin{equation}\label{gapN}
N(t;F_{k})-N_0(t;F_{k})\leq \frac{\theta_{F}(t)}{\pi}+S(t;F_{k})-N_0(t;F_{0})+O_k(1).
\end{equation}
Since the argument in the previous section applies to the case $k=0$ with minor modifications, we obtain
\begin{equation}\label{zetazero}
N(t;F_{0})=\frac{\theta_{F}(t)}{\pi}+S(t;F_{0})+O(1),
\end{equation}
where
\begin{equation*}
S(t;F_{0})=\frac{1}{\pi}\left\{ \int_{2}^{2+it}+\int_{2+it}^{1/2+it}\right\}d\arg F(s).
\end{equation*}
Under RH, we have $N(t;F_{0})=N_0(t;F_{0})$.
Therefore, by \eqref{gapN} and \eqref{zetazero}, we have
\begin{equation}\label{gapNt}
N(t;F_{k})-N_0(t;F_{k})\leq S(t;F_{k})-S(t;F_{0})+O_k(1).
\end{equation}
Now let $0<T<T'$.
Since $N(t;F_{k})-N_0(t;F_{k})$ is non-negative and increasing, we obtain
\begin{align*}
&\quad\int_{0}^{T'}(N(t;F_{k})-N_0(t;F_{k}))\, dt \\
&\geq \int_{T}^{T'}(N(t;F_{k})-N_0(t;F_{k}))\, dt \\
&\geq (T'-T)(N(T;F_{k})-N_0(T;F_{k})).
\end{align*}
Substituting~\eqref{gapNt} into the above inequality and using~\eqref{StF_{k}},
we obtain
\begin{equation*}
N(T;F_{k})-N_0(T;F_{k})\leq \frac{1}{T'-T}\{O_k(T')+O(T')+O_k(T') \}.
\end{equation*}
Letting $T'\to \infty$, we conclude that
\begin{equation*}
N(T;F_{k})-N_0(T;F_{k})=O_k(1).
\end{equation*}
This completes the proof.

\section{Proofs of Theorems~\ref{th_EMF} and~\ref{th_main}}

To prove Theorem~\ref{th_EMF}, we make use of the entire fucntion $\xi_{F,k}(s)$ constructed eariler.

We next establish a Hadamard factorisation of $\xi_{F,k}(s)$.
\begin{prop}\label{factorisation}
  For each $k\ge 0$, there are constants $A_{k}$ and $B_{k}$ such that
  \[
  \xi_{F,k}(s)=e^{A_{k}+B_{k}s}\prod_{\rho_{k}}\left( 1-\frac{s}{\rho_{k}} \right)e^{\frac{s}{\rho_{k}}}
  \]
  for all $s$.
  Here the product is extended over all zeros $\rho_{k}$ of $\xi_{F,k}$.
\end{prop}

\begin{proof}
  By the Hadamard factorisation theorem, it suffices to show that the order of $\xi_{F,k}(s)$ is $1$.
  Let $\sigma\ge 1/2$.
  By the reflection formula, we have
  \[
  \Gamma(\lambda_{j}(1-s)+\overline{\mu_{j}})=\frac{\pi}{\Gamma(\lambda_{j}(s-1)+1-\overline{\mu_{j}})\sin \pi(\lambda_{j}(1-s)+\overline{\mu_{j}})}.
  \]
  It follows that
  \[
  \xi_{F,k}(s)=s^{m_{F}}(s-1)^{m_{F}}Q^{s}\prod_{j=1}^{r}\frac{\Gamma(\lambda_{j}(s-1)+1-\overline{\mu_{j}})^k \sin^{k}\pi(\lambda_{j}(1-s)+\overline{\mu_{j}})}{\pi^{k}\Gamma(\lambda_{j}s+\mu_{j})^{k-1}}F_{k}(s).
  \]
  Note that $\sin \pi(\lambda_{j}(1-s)+\overline{\mu_{j}})$ has zeros at $s=1+\frac{\overline{\mu_{j}}\mp n}{\lambda_{j}}$ and satisfies
  \[
  |\sin \pi(\lambda_{j}(1-s)+\overline{\mu_{j}})|\le e^{\pi\lambda_{j}|t|}.
  \]
  Therefore, together with Stirling's formula, for $\sigma \ge 1/2$, $ \xi_{F,k}(s) \ll e^{C|s|\log|s|}$ with a positive constant $C$.
  By the functional equation~\eqref{xi_FE}, the same estimate holds for $\sigma<1/2$.
  Hence, for sufficiently large $|s|$, we obtain
  \[
  \xi_{F,k}(s) \ll e^{C|s|\log|s|}.
  \]
  The above estimate shows that $\xi_{F,k}(s)$ is an entire function of finite order at most $1$.
  On the other hand, since $F_k(s)$ is a finite linear combination of derivatives of $F(s)$, we have
  \[
  F_k(\sigma)=O(1)
  \quad (\sigma\to \infty).
  \]
  Hence, by another application of Stirling's formula,
  \[
  \xi_{F,k}(\sigma)\sim e^{\frac{d_{F}}{2}\sigma\log \sigma+O(\sigma)}
  \quad (\sigma\to \infty).
  \]
  This completes the proof.
\end{proof}

We now turn to the proof of Theorem~\ref{th_EMF}.
By the definition of $\xi_{F,k}(s)$ and Proposition \ref{ZandF}, we have
\begin{align*}
  &\quad
  \xi_{F,k}\left( \frac{1}{2}+it \right) \\
  &=
  (-1)^{m_{F}}i^{-k}\left( \frac{1}{4}+t^2 \right)^{m_{F}}\left( \frac{Q}{\omega} \right)^{\frac{1}{2}}\prod_{j=1}^{r}\left| \Gamma\left( \frac{\lambda_{j}}{2}+\mu_{j}+i\lambda_{j}t \right) \right|^{2k+1}Z_{F}^{(k)}(t).
\end{align*}
Hence when we put
\[
g_{F,k}(t)
=
(-1)^{m_{F}}i^{-k}\left( \frac{1}{4}+t^2 \right)^{m_{F}}\left( \frac{Q}{\omega} \right)^{\frac{1}{2}}\prod_{j=1}^{r}\left| \Gamma\left( \frac{\lambda_{j}}{2}+\mu_{j}+i\lambda_{j}t \right) \right|^{2k+1}
\]
then, by the logarithmic derivative with respect to $t$, we obtain
\[
i\frac{\xi_{F,k}'}{\xi_{F,k}}\left( \frac{1}{2}+it \right)=\frac{g_{F,k}'}{g_{F,k}}(t)+\frac{Z_{F}^{(k+1)}}{Z_{F}^{(k)}}(t).
\]
As for the function $(g_{F,k}'/g_{F,k})(t)$, it is clear that
\[
\frac{g_{F,k}'}{g_{F,k}}(t)
=
\frac{8mt}{1+4t^2}-(2k-1)\sum_{j=1}^{r}\frac{d}{dt}\Re \log \Gamma\left( \frac{\lambda_{j}}{2}+\mu_{j}+i\lambda_{j}t \right)
\]
and hence
\[
\frac{d}{dt}\frac{g_{F,k}'}{g_{F,k}}(t)\ll t^{-1}.
\]

On the other hand, by Proposition~\ref{factorisation}, we have
\[
\frac{\xi_{F,k}'}{\xi_{F,k}}(s)
=
B_{k}+\sum_{\rho_{k}}\left( \frac{1}{s-\rho_{k}}+\frac{1}{\rho_{k}} \right).
\]
Therefore
\begin{align*}
  \frac{d}{dt}\frac{\xi_{F,k}'}{\xi_{F,k}}\left( \frac{1}{2}+it \right)
  &=
  \sum_{\rho_{k}}\frac{-i}{(\frac{1}{2}+it-\rho_{k})^2} \\
  &=
  i\sum_{\gamma_{k}}\frac{1}{(t-\gamma_{k})^2}+\sum_{\substack{\rho_{k} \\ \beta_{k}\neq \frac{1}{2}}}\frac{-i}{(\frac{1}{2}+it-\rho_{k})^2} \\
  &=
  i\sum_{\gamma_{k}}\frac{1}{(t-\gamma_{k})^2} \\
  &\quad
  +\sum_{\substack{\rho_{k} \\ \beta_{k}<1-\sigma_{F,k}, \sigma_{F,k}<\beta_{k}}}\frac{-i}{(\frac{1}{2}+it-\rho_{k})^2}+O_{k}(t^{-2}).
\end{align*}
Following the argument of Matsuoka~\cite[p.~15]{M}, we have
\[
\sum_{\substack{\rho_{k} \\ \beta_{k}<1-\sigma_{F,k}, \sigma_{F,k}<\beta_{k}}}\frac{1}{(\frac{1}{2}+it-\rho_{k})^2}
\ll_{k}
\sum_{n=0}^{\infty}\frac{1}{(t+m)^2}
\ll_{k}
\int_{0}^{\infty}\frac{dx}{(t+x)^2}
\ll
\frac{1}{t}.
\]
Thus we have
\[
\frac{d}{dt}\frac{\xi_{F,k}'}{\xi_{F,k}}\left( \frac{1}{2}+it \right)
=
i\sum_{\gamma_{k}}\frac{1}{(t-\gamma_{k})^2}+O_{k}(t^{-1}).
\]
This yields
\[
\frac{d}{dt}\frac{Z_{F}^{(k+1)}}{Z_{F}^{(k)}}(t)
=
-\sum_{\gamma_{k}}\frac{1}{(t-\gamma_{k})^2}+O_{k}(t^{-1}),
\]
which is the desired formula.

By Theorem~\ref{th_EMF}, we have
\begin{align*}
  \frac{d}{dt}\frac{Z_{F}^{(k+1)}}{Z_{F}^{(k)}}(t)
  &<
  -\sum_{0<\gamma_{k}<t}\frac{1}{(t-\gamma_{k})^2}+At^{-1} \\
  &<
  t^{-1}(A-N_{0}(T;F_{k})t^{-1})
\end{align*}
for some constant $A$.
From Theorem~\ref{th_countingzeros} and Theorem~\ref{th_RH}, this is negative for large $t$.
Therefore, $(Z_{F}^{(k+1)}/Z_{F}^{(k)})(t)$ monotonically decreases between any two consecutive zeros of $Z_{F}^{(k)}(t)$ for large $t$.
Consequently, $(Z_{F}^{(k+1)}/Z_{F}^{(k)})(t)$ has exactly one zero between any two consecutive zeros of $Z_{F}^{(k)}(t)$ for large $t$, and the same holds for $Z_{F}^{(k+1)}(t)$.
This completes the proof of Theorem~\ref{th_main}.

\section*{Acknowledgements}
This work was supported by JSPS KAKENHI Grant Number 25K17245.

\end{document}